\documentclass{amsart}
\usepackage{amsmath}%
\usepackage{amsfonts}%
\usepackage{amssymb}%
\usepackage{graphicx}
\usepackage{color}
\usepackage{verbatim}   
\theoremstyle{plain}

\newtheorem{theorem}{Theorem}[section]
\newtheorem{lemma}[theorem]{Lemma}
\newtheorem{corollary}[theorem]{Corollary}

\newtheorem{definition}[theorem]{Definition}
\newtheorem{example}[theorem]{Example}
\newtheorem{remark}[theorem]{Remark}
\newtheorem{proposition}[theorem]{Proposition}
\newtheorem{assumption}[theorem]{Assumption}
\newtheorem{claim}[theorem]{Claim}

\numberwithin{equation}{section}

\usepackage[colorlinks, linkcolor = black, citecolor = black, filecolor = black, urlcolor = blue,]{hyperref}

\renewcommand{\div}{\operatorname{div}\!}

\DeclareMathOperator{\sgn}{sgn}

\renewcommand{\Im}{\operatorname{\mathfrak{Im}}}
\renewcommand{\Re}{\operatorname{\mathfrak{Re}}}

\newcommand{\supp}{\operatorname{supp}}

\newcommand{\bnd}{\operatorname{\mathcal{L}}}

\newcommand{\doms}[2]{\textbf{D}_{#1}\left(#2\right)} 

\newcommand{\kers}[2]{\textbf{N}_{#1}\left(#2\right)}

\newcommand{\ran}[1]{\textbf{R}\left(#1\right)}

\newcommand{\rans}[2]{\textbf{R}_{#1}\left(#2\right)} 

\newcommand{\crans}[2]{\overline{\textbf{R}_{#1}\left(#2\right)}}

\newcommand{\proj}[1]{\mathbb{P}_{#1}}

\newcommand{\bisec}[1]{S_{#1}}

\newcommand{\opbisec}[1]{\dot{S}_{#1}}

\newcommand{\opsec}[2]{\dot{S}_{#2}^{#1}}

\newcommand{\Hinf}[1]{H^{\infty}\big(\opbisec{#1}\big)}

\newcommand{\Hinfzer}[1]{H^{\infty}\big(\opbisec{#1},\{0\}\big)}

\newcommand{\Psidec}[1]{\Psi\big(\opbisec{#1}\big)}

\newcommand{\Psidecpar}[3]{\Psi_{#1}^{#2}\big(\opbisec{#3}\big)}

\newcommand{\LR}[2]{L^{#1}\left(\mathbb{R}^{#2}\right)}
\newcommand{\LRC}[3]{L^{#1}\left(\mathbb{R}^{#2},\mathbb{C}^{#3}\right)}

\def\Xint#1{\mathchoice
{\XXint\displaystyle\textstyle{#1}}%
{\XXint\textstyle\scriptstyle{#1}}%
{\XXint\scriptstyle\scriptscriptstyle{#1}}%
{\XXint\scriptscriptstyle\scriptscriptstyle{#1}}%
\!\int}
\def\XXint#1#2#3{{\setbox0=\hbox{$#1{#2#3}{\int}$ }
\vcenter{\hbox{$#2#3$ }}\kern-.6\wd0}}

\def\dashint{\Xint-}

\begin{document}
\title[$L^{p}-L^{q}$ theory]{$L^{p}-L^{q}$ theory for holomorphic functions of perturbed first order Dirac operators}
\author{Sebastian Stahlhut}
\address{Univ. Paris-Sud, laboratoire de Math\'ematiques, UMR 8628 du CNRS, F-91405 {\sc Orsay}}%
\email{sebastian.stahlhut@math.u-psud.fr}%

\date{September 9,  2014}

\begin{abstract}
The aim of the article is to prove $L^{p}-L^{q}$ off-diagonal estimates and $L^{p}-L^{q}$ boundedness for operators in the functional calculus of certain perturbed first order differential operators of Dirac type for  with $p\le q$ in a certain range of exponents. We describe the $L^{p}-L^{q}$ off-diagonal estimates and the $L^{p}-L^{q}$ boundedness in terms of the decay properties of the related holomorphic functions and give a necessary condition for $L^{p}-L^{q}$ boundedness. Applications to Hardy-Littlewood-Sobolev estimates for fractional operators will be given.
\end{abstract}
\maketitle

\section{Introduction}
In this article, we are interested in $L^{p}-L^{q}$ estimates for operators defined by the functional calculus of certain first order differential operators of Dirac type. Let us start with an example. In one dimension $\left(Id-it\frac{d}{dx}\right)^{-1}$, where $-i\frac{d}{dx}$ is defined as a self-adjoint operator in $\LR{2}{}$ and $t\in \mathbb{R}$, is known to have a kernel $\frac{1}{2|t|}e^{-\frac{|x-y|}{|t|}}$, hence it is bounded from $\LR{p}{}$ to $\LR{q}{}$ for any $1\leq p\leq q\leq \infty$. In higher dimensions the operators $\left(Id+itD\right)^{-1}$, where $\displaystyle D=\begin{pmatrix} 0 & \div \\ -\nabla & 0 \end{pmatrix}$, are examples of bounded operators on $\LRC{p}{n}{N}$ but not from $\LRC{p}{n}{N}$ to $\LRC{q}{n}{N}$ $p,q\in \left(1,\infty\right)$ with $p< q$. However, the operator $tD\left(Id+itD\right)^{-2}$ satisfies for all Borel sets $E,F\subset \mathbb{R}^{n}$, all $u\in \LRC{p}{n}{N}$ and for certain values $K\in\left[0,\infty\right)$ and $q>p$, 
\begin{align*}
||tD\left(Id+itD\right)^{-2}\chi_{E}u||_{L^{q}\left(F\right)}\lesssim |t|^{\frac{n}{q}-\frac{n}{p}}\left(1+\frac{d\left(E,F\right)}{|t|}\right)^{-K} ||u||_{L^{p}\left(E\right)},
\end{align*}
where $d\left(E,F\right):=\inf\left\{|x-y||x\in E,\ y\in F\right\}$ is the distance between the sets $E$ and $F$ and $\chi_{E}$ denotes the characteristic function of $E$.

Here, we want to explore this phenomenon for perturbed first order Dirac operators $DB$ and $BD$ (see below for definitions). The off-diagonal estimates are important, when one seeks to prove, for example, 
\begin{align*}
\left(\int_{\mathbb{R}^{n}}\left(\int_{0}^{\infty}\int_{B\left(x,t\right)}|tDB\left(Id+itDB\right)^{-2}u\left(y\right)|^{2}\frac{dydt}{t^{n+1}}\right)^{p/2}dx\right)^{\frac{1}{p}}\lesssim ||u||_{L^{p}}
\end{align*}
for certain values of $p$, where $B\left(x,t\right)\subset \mathbb{R}^{n}$ denotes the ball of radius $t$ and center $x$. In fact, this will be the application in the subsequent work \cite{Aus-Sta-2} of Pascal Auscher and the author.

The notion of $L^{2}-L^{2}$ off-diagonal estimates arises from \cite{Dav} and \cite{Gaf}. Such off-diagonal estimates were proved and used for second order elliptic operators for the solution of the Kato square root problem in \cite{AHLMT} and used to compensate the  lack of pointwise kernel estimates. In \cite{Hof-Mar} such $L^{2}-L^{2}$ off-diagonal estimates were used to prove certain $L^{p}$ bounds Riesz transform associated to second order elliptic operators. It was in \cite{Blu-Kun} that  the  $L^p-L^q$ version of those were used to prove $L^p$ estimates in absence of pointwise bounds.  In \cite{Blu-Kun}, \cite{Aus-2}, \cite{Hof-May-McI}, \cite{Hof-May}, \cite{Aus-Hof-Mar} and \cite{Aus-1}, $L^{p}-L^{q}$ off-diagonal estimates for semigroup and resolvent of elliptic second order differential operators were used to prove square root estimates, boundedness for square functions, certain maximal functions, Riesz transforms, ect. Similar work was done in the context of self-adjoint operators in \cite{Hof-Lu-Mit-Yan}, \cite{Aus-McI-Mor} and \cite{Bui-Duo}. Another line of developments was a generalized Calder\'on-Zygmund theory for operators which do not satisfy kernel estimates, where one used off-diagonal estimates as replacement (cf. \cite{Blu-Kun}, \cite{Aus-1}, \cite{Aus-Mar-1}, \cite{Aus-Mar-2}, \cite{Ber}, \cite{Fre-Kun}, \cite{Ber-Fre} etc.). As consequence of the solution of the Kato square root problem, interest arises also in square function estimates for first order Dirac operators. In \cite{Axe-Kei-McI} and \cite{Aus-Axe-McI-2} vertical square function estimates were proved using $L^{2}-L^{2}$ off-diagonal estimates for the resolvent. In \cite{Aus-McI-Rus} such off-diagonal estimates for the resolvent were used to develop a Hardy space theory associated to Hodge-Dirac operators on manifolds. In \cite{Hyt-McI} and \cite{Hyt-McI-Por} $L^{p}-L^{p}$ off-diagonal estimates for the resolvent of first order Dirac operators were applied to prove an extrapolation theorem for $R$-bisectoriality and to prove equivalence of $R$-bisectoriality and holomorphic functional calculus on intervals of Lebesgue exponents. \cite{Aji} even introduced an $L^{p}-L^{q}$ theory for first order Dirac operators under certain restrictions, which we remove here.

Our plan is as follows. In the second section, we introduce the basic notions. In the third section, we discuss our main results. In the first part, we give sufficient conditions for $L^{p}-L^{q}$ off-diagonal estimates and $L^{p}-L^{q}$ boundedness for operators in the functional calculus of these perturbed first order Dirac operators in terms of decay properties at $0$ and $\infty$ for the associated holomorphic functions. In particular, we give a relation between the decay properties for the associated holomorphic functions and the number $K$ above. These results will be given in Propositions \ref{proposition3.3} and \ref{proposition3.8} below. Corollary \ref{corollary3.13} gives a version for bounded holomorphic functions, which have no decay at $0$. These results are partially contained in the work of \cite{Aji} when the range of the perturbed first order Dirac operator is stable under multiplication by smooth cut-off functions/cut-off functions. In the second part of Section 3, we discuss when this is the case. We give a condition in Proposition \ref{proposition..3.18} that shows that for the operators $D$ and $DB$, Ajiev's results may not be always applicable, in particular not for $\displaystyle D=\begin{pmatrix} 0 & \div \\ -\nabla & 0 \end{pmatrix}$ as above, whereas ours are. In the third part of Section 3, we give a necessary condition for $L^{p}-L^{q}$ boundedness when $p<q$, which highlights the connection of $L^{p}-L^{q}$ boundedness to kernel/range decompositions. In particular, this condition shows that the operators $\left(Id+itD\right)^{-1}$, $\left(Id+itDB\right)^{-1}$ and $\left(Id+itBD\right)^{-1}$ are not $L^{p}-L^{q}$ bounded for the particular $\displaystyle D=\begin{pmatrix} 0 & \div \\ -\nabla & 0 \end{pmatrix}$ and $p<q$. Also, this condition shows that the semi-groups $e^{-t|D|}$, $e^{-t|DB|}$ and $e^{-t|BD|}$, where $|.|=\sqrt{\left(.\right)^{2}}$, are not $L^{p}-L^{q}$ bounded for this particular $D$. The last part of Section 3 concers analytic extensions of our results to complex times $t$. Finally, we give an application of $L^{p}-L^{q}$ boundedness to estimates for fractional operators related to $D$, $DB$ and $BD$ in Section 4.

\section{Setting}
\subsection{Definitions and Notation}
Let $1<q<\infty$, $n,N\in\mathbb{N}^{*}$. By an unbounded operator on $L^{q}:=\LRC{q}{n}{N}$ we mean a linear map $T:\doms{q}{T}\rightarrow L^{q}$ with domain $\doms{q}{T}\subset L^{q}$. We denote the null space by $\kers{q}{T}$ and the range by $\rans{q}{T}$. We say that $T$ admits a kernel/range decomposition in $L^{q}$ whenever 
\begin{align}\label{EQ.2.0}
L^{q}=\kers{q}{T}\oplus\crans{q}{T},
\end{align}
where the sum is topological and $\crans{q}{T}$ is the closure of $\rans{q}{T}$ in $L^{q}$. A class of operators which admit a kernel/range decomposition are bisectorial operators. We say that a linear operator $T$ is bisectorial of type $\omega\in \left[0,\frac{\pi}{2}\right)$ if $T$ is closed, the spectrum of $T$ is contained in a bisector $\bisec{\omega}:=\left\{\lambda\in \mathbb{C}\backslash\left\{0\right\}:|\arg \lambda|\leq \omega \text{ or } |\arg \left(-\lambda\right)|\leq \omega\right\}\cup\left\{0\right\}$ and for each $\nu \in \left(\omega,\frac{\pi}{2}\right)$ there exists a constant $C_{\nu}>0$ such that
\begin{align}\label{EQ.2.1}
||\left(Id+\lambda T\right)^{-1}||_{L^{q}\rightarrow L^{q}}\leq C_{\nu}
\end{align}
for all $\lambda \in \mathbb{C}\backslash \bisec{\nu}$.
The bound (\ref{EQ.2.1}) allows one to define a functional calculus. To $\sigma>0$, $\tau>0$ and $\nu\in \left(0,\frac{\pi}{2}\right)$ we define $\Psidecpar{\sigma}{\tau}{\nu}$ to be the set of all holomorphic functions $\psi:\opbisec{\nu}\rightarrow \mathbb{C}$ such that 
\begin{align}\label{EQ.2.2}
|\psi\left(\lambda\right)|\lesssim \frac{|\lambda|^{\sigma}}{1+|\lambda|^{\sigma+\tau}}.
\end{align}
for all $\lambda\in \opbisec{\nu}:=\left\{\lambda\in \mathbb{C}\backslash\left\{0\right\}:|\arg \lambda|< \omega \text{ or } |\arg \left(-\lambda\right)|< \omega\right\}$.
Moreover, we define $\Psidec{\nu}:=\bigcup_{\sigma,\tau>0}\Psidecpar{\sigma}{\tau}{\nu}$ and set $\psi\left(0\right)=0$ when $\psi \in \Psidec{\nu}$. Having these definitions in hand we can define a functional calculus as follows. Let $T$ be a bisectorial operator of type $\omega\in \left[0,\frac{\pi}{2}\right)$ on $L^{q}$, then for each $\nu\in \left(\omega,\frac{\pi}{2}\right)$ the Dunford integral 
\begin{align}
\psi\left(T\right)&:=\frac{1}{2\pi i}\int_{\partial \bisec{\theta}}\psi\left(\lambda\right)\left(Id-\lambda^{-1}T\right)^{-1}\frac{d\lambda}{\lambda} \label{EQ.2.3}
\end{align}
defined as improper Riemann integral converges normally in $\bnd\left(L^{q}\right)$ for all $\theta\in \left(\omega,\nu\right)$, where $\partial \bisec{\theta}:=\left\{\pm te^{\pm i\theta}:t\in \left(0,\infty\right)\right\}$ is oriented counterclockwise on the four branches surrounding $S_{\omega}$. We say that a bisectorial operator $T$ of type $\omega$ has a bounded holomorphic functional calculus, if for each $\nu\in\left(\omega,\frac{\pi}{2}\right)$ there exists a constant $C_{\nu}>0$ such that for all $\psi\in\Psidec{\nu}$ and all $u\in L^{q}$ holds
\begin{align*}
||\psi\left(T\right)u||_{L^{q}}\leq C_{\nu}||\psi||_{\Hinf{\nu}}||u||_{L^{q}}.
\end{align*}
Whenever this is the case the bounded holomorphic functional calculus may be extended to the class $\Hinf{\nu}$ by a limiting procedure and to the class $\Hinfzer{\nu}$ using the kernel/range decomposition of $T$. 
Indeed in this case, for $f\in \Hinfzer{\nu}$ we define 
\begin{align}\label{EQ.2.4}
f\left(T\right)u:=f\left(0\right)u_{N}+f|_{\opbisec{\nu}}\left(T\right)u_{R}
\end{align}
where $u_{N},u_{R}$ denote the projections of $u$ onto null space and closure of the range of $T$ according to (\ref{EQ.2.0}).
Here, $\Hinf{\nu}$ is the set of all bounded holomorphic functions $f:\opbisec{\nu}\rightarrow \mathbb{C}$ with  norm $||f||_{\Hinf{\nu}}:=\sup_{z\in \opbisec{\nu}}|f\left(z\right)|$ and  $\Hinfzer{\nu}$ is the set of all bounded functions $f:\opbisec{\nu}\cup\left\{0\right\}\rightarrow \mathbb{C}$ with norm $||f||_{\Hinfzer{\nu}}:=\sup_{z\in \opbisec{\nu}\cup\left\{0\right\}}|f\left(z\right)|$ such that the restriction $f|_{\opbisec{\nu}}$ is holomorphic. For more details on kernel/range decompositions, bisectorial operators and functional calculus we refer the reader to \cite{Cow-Dou-McI-Yag}, \cite{LeM}, \cite{Haa}, \cite{McI} and the references therein.
\newline
We fix $p,q$ with $1<p\leq q<\infty$ in the sequel. In the following we are interested in the following three boundedness properties for families of operators.

\begin{definition}[Boundedness for families of operators]
Let $\mathcal{A}\subset \mathbb{C}\backslash \left\{0\right\}$ be a subset of the complex plane and $U_{p}\subset L^{p}$, $U_{q}\subset L^{q}$ be closed subspaces. We say that a family of operators $\left\{T_{\lambda}\right\}_{\lambda\in\mathcal{A}}$ is $U_{p}-U_{q}$ bounded if for all $\lambda\in \mathcal{A}$ and all $u\in U_{p}$ holds $T_{\lambda}u\in U_{q}$ and there exists a constant $C_{p,q}>0$ such that for all $\lambda\in \mathcal{A}$ and all $u\in U_{p}$ holds 
\begin{align*}
||T_{\lambda}u||_{L^{q}}\leq C_{p,q} |\lambda|^{\frac{n}{q}-\frac{n}{p}}||u||_{L^{p}}.
\end{align*}
\end{definition}

\begin{definition}[Off-diagonal estimates for families of operators]
Let $\mathcal{A}\subset \mathbb{C}\backslash \left\{0\right\}$ be a subset of the complex plane. We say that a family of operators $\left\{T_{\lambda}\right\}_{\lambda\in\mathcal{A}}$ satisfies $L^{p}-L^{q}$ off-diagonal estimates of order $K\in \left[0,\infty\right)$ if there exists a constant $C_{K,p,q}>0$ such that for all Borel sets $E,F\subset \mathbb{R}^{n}$, all $\lambda\in \mathcal{A}$ and all $u\in L^{p}$ holds 
\begin{align*}
||\chi_{E}T_{\lambda}\left(\chi_{F}u\right)||_{L^{q}}\leq C_{K,p,q} |\lambda|^{\frac{n}{q}-\frac{n}{p}}\left(1+\frac{d\left(E,F\right)}{|\lambda|}\right)^{-K}||\chi_{F}u||_{L^{p}},
\end{align*}
where $d\left(E,F\right):=\inf_{x\in E, y\in F}|x-y|_{\mathbb{R}^{n}}$ denotes the distance between $E$ and $F$ and $\chi_{E}$ denote the characteristic functions of a set $E$.
\end{definition}

\begin{definition}[Biparameter off-diagonal estimates for families of operators]
Let $\mathcal{A},\mathcal{B}\subset \mathbb{C}\backslash \left\{0\right\}$ be two subsets of the complex plane. We say that a family of operators $\left\{T_{\lambda_{1},\lambda_{2}}\right\}_{\left(\lambda_{1},\lambda_{2}\right)\in\mathcal{A}\times \mathcal{B}}$ satisfies $L^{p}-L^{q}$ biparameter off-diagonal estimates in $\left(\lambda_{1},\lambda_{2}\right)$ of order $K\in \left[0,\infty\right)$ if there exists a constant $C_{K,p,q}>0$ such that for all Borel sets $E,F\subset \mathbb{R}^{n}$, all $\left(\lambda_{1},\lambda_{2}\right)\in\mathcal{A}\times \mathcal{B}$ and all $u\in L^{p}$ holds 
\begin{align*}
||\chi_{E}T_{\lambda_{1},\lambda_{2}}\left(\chi_{F}u\right)||_{L^{q}}\leq C_{K,p,q} |\lambda_{1}|^{\frac{n}{q}-\frac{n}{p}}\left(1+\frac{d\left(E,F\right)}{|\lambda_{2}|}\right)^{-K}||\chi_{F}u||_{L^{p}}.
\end{align*}
\end{definition}

\subsection{First order Dirac operators}
We are interested in families of operators defined by the bounded holomorphic functional calculus of the following special class of bisectorial operators.

\begin{assumption}\label{Assumption2.4}
Let $n,N \in \mathbb{N}^{*}$. Let $D$ be a first order differential operator on $\mathbb{R}^{n}$ acting on functions valued in $\mathbb{C}^{N}$ that satisfies the conditions (D0), (D1) and (D2) in \cite{Hyt-McI}. These are
\begin{itemize}
\item[(1)] $D$ has the representation $D=-i\sum_{j=1}^{n}\hat{D}_{j}\partial_{j}$ with matrices $\hat{D}_{j}\in \mathcal{L}\left(\mathbb{C}^{N}\right)$,
\item[(2)] There exists $\kappa>0$ such that the symbol $\hat{D}\left(\xi\right)=\sum_{j=1}^{n}\hat{D}_{j}\xi_{j}$ satisfies $\kappa|\xi||e|\leq |\hat{D}\left(\xi\right)e|$ for all $\xi\in \mathbb{R}^{n}$ and all $e\in \ran{\hat{D}\left(\xi\right)}$,  
\item[(3)] There exists $\omega_{D}\in \left[0,\frac{\pi}{2}\right)$ such that the spectrum of the symbol satisfies $\sigma\left(\hat{D}\left(\xi\right)\right)\subset \bisec{\omega_{D}}$.
\end{itemize}
Further, let $B$ be the operator defined via pointwise multiplication by the matrix function $B\left(x\right)$, $x\in \mathbb{R}^{n}$, with $B\in L^{\infty}\left(\mathbb{R}^{n},\mathcal{L}\left(\mathbb{C}^{N}\right)\right)$. We assume additionally one of the following equivalent conditions: 
\begin{enumerate}
\item[(4)] Assume $B$ satisfies the coercivity condition $||Bu||_{L^{2}}\gtrsim ||u||_{L^{2}}$ for all $u\in \crans{2}{D}$ and there exists $\omega\in \left[0,\frac{\pi}{2}\right)$ such that $BD$ is bisectorial of type $\omega$ on $L^{2}$,
\item[(5)] Assume $B^{*}$ satisfies the coercivity condition $||B^{*}u^{*}||_{L^{2}}\gtrsim ||u^{*}||_{L^{2}}$ for all $u^{*}\in \crans{2}{D^{*}}$ and there exists $\omega\in \left[0,\frac{\pi}{2}\right)$ such that $DB$ is bisectorial of type $\omega$ on $L^{2}$.
\end{enumerate}
\end{assumption}
\textbf{In the sequel, we shall systematically assume without mention that Assumption \ref{Assumption2.4} holds in all statements involving $DB$ or $BD$.}
\begin{example}
The operators $D$ and $B$ appearing in the works \cite{Aus-Axe-McI-1} and \cite{Aus-Axe} satisfy Assumption \ref{Assumption2.4}. Further examples are in \cite{Hyt-McI-Por} and \cite{Hyt-McI}. 
\end{example}

Here, we do not assume that $D$ is self-adjoint or that $B$ satisfies a strictly accretivity condition. The equivalence of the conditions (4) and (5) was proven in \cite{Aus-Sta-1}. The first consequence of Assumption \ref{Assumption2.4} is the following proposition due to \cite{Hyt-McI-Por}.
\begin{proposition}\label{proposition2.9} 
Let $1<q<\infty$.
\begin{enumerate}
  \item \label{Enu1prop2.9} $D$ is a bisectorial operator of type $\omega_{D}$ with bounded holomorphic functional calculus in $L^{q}$.
\item \label{Enu2prop2.9} $L^{q}=\kers{q}{D}\oplus \crans{q}{D}$, i.e. $D$ admits a kernel/range decomposition on $\LRC{q}{n}{N}$,
  \item \label{Enu3prop2.9} $\kers{q}{D}$ and $\crans{q}{D}$, $1<q<\infty$, are complex interpolation families.
  \item \label{Enu4prop2.9} $D$ satisfies the coercivity condition
\begin{align*}
  ||\nabla u ||_{L^{q}\left(\mathbb{R}^{n},\mathbb{C}^{n}\otimes\mathbb{C}^{N}\right)} \lesssim ||Du||_{L^{q}\left(\mathbb{R}^{n},\mathbb{C}^{N}\right)}\quad \text{ for all } u \in \doms{q}{D}\cap \crans{q}{D} \subset W^{1,q}. 
\end{align*}  
Here, we use the notation $\nabla u$ for $\nabla \otimes u$ and $||u||_{W^{1,q}}=||u||_{L^{q}}+||\nabla u||_{L^{q}}$. 
\item \label{Enu5prop2.9} The same properties hold for the adjoint $D^{*}$.
\end{enumerate}
\end{proposition}
By \cite{Hyt-McI} and \cite[Proposition 2.1]{Aus-Sta-1}, it follows that the operators $BD$ and $DB$ have a meaning as unbounded operators in $L^{q}$ with natural domains $\doms{q}{D}$ and $B^{-1}\doms{q}{D}$, the preimage of $\doms{q}{D}$ under $B$.
Moreover, Assumption \ref{Assumption2.4} implies the existence of an open interval $\mathcal{I}_{2}\subset \left(1,\infty\right)$ containing $2$ such that for all $q\in \mathcal{I}_{2}$ holds $||Bu||_{L^{q}}\geq C||u||_{L^{q}}$ whenever $u\in \crans{q}{D}$ and $||B^{*}u^{*}||_{L^{q'}}\geq C||u^{*}||_{L^{q'}}$ whenever $u^{*}\in \crans{q'}{D^{*}}$. This was shown in \cite{Aus-Sta-1} and \cite{Hyt-McI} and used to extrapolate $R$-bisectoriality. As $L^{2}$-bisectoriality self-improves to $L^{2}$-$R$-bisectoriality\footnote{We do not introduce this notion here, as it is not of further interest.} we get from the works \cite{Aus-Sta-1}[Theorem 5.1], \cite[Lemma 2.4, Theorem 2.5]{Hyt-McI}, \cite[Corollary 8.17]{Hyt-McI-Por} and \cite[Theorem 5.3]{Kal-Wei} (to recall the main ingredients) the following theorem.
 
\begin{theorem}\label{theorem2.5}
There exists an open interval $\mathcal{I}_{D,B}=\left(p_{-}\left(D,B\right),p_{+}\left(D,B\right)\right)$ containing $2$ and maximal in $\mathcal{I}_{2}$ such that the following properties for $T\in \left\{BD,DB\right\}$ and $q\in \mathcal{I}_{D,B}$ hold:
\begin{enumerate}
\item \label{Enu2thm2.5} $T$ admits a kernel/range decomposition on $\LRC{q}{n}{N}$,
\item \label{Enu3thm2.5} $T$ is a bisectorial operator of type $\omega$ on $\LRC{q}{n}{N}$,
\item \label{Enu4thm2.5} $T$ has a bounded holomorphic functional calculus on $\LRC{q}{n}{N}$,
\item \label{Enu5thm2.5} for each $\nu \in \left(\omega,\frac{\pi}{2}\right)$ the family $\left\{\left(Id+\lambda T\right)^{-1}\right\}_{\lambda\in \mathbb{C}\backslash \bisec{\nu}}$ satisfies $L^{q}-L^{q}$ off-diagonal estimates of order $K$ for all $K\in \left[0,\infty\right)$.
\end{enumerate}
Moreover, all the properties (\ref{Enu2thm2.5}), (\ref{Enu3thm2.5}), (\ref{Enu4thm2.5}) and (\ref{Enu5thm2.5}) fail whenever $q=p_{\pm}\left(D,B\right)\in \mathcal{I}_{2}$.
\end{theorem} 
For a more complete and more general version of this theorem we refer the reader to 
\cite{Sta}. Finally let us make a few remarks in relation to Theorem \ref{theorem2.5}.

\begin{remark}\label{rem2.6}
For $T\in \left\{BD,DB\right\}$ and $q\in \mathcal{I}_{D,B}$ let $T_{q}$ be the $L^{q}$-realization. Then 
\begin{enumerate}
\item $\left(BD\right)_{q}=BD$ with domain $\doms{q}{BD}=\doms{q}{D}$.
\item $\left(DB\right)_{q}=DB$ with domain $\doms{q}{DB}=B^{-1}\doms{q}{D}$.
\item $T_{p}=T_{q}$ on $\doms{p}{T}\cap\doms{q}{T}$.
\item $f\left(T_{p}\right)=f\left(T_{q}\right)$ on $L^{p}\cap L^{q}$ for $f\in \Hinfzer{\nu}$.
\item $\mathbb{P}_{\crans{p}{T}}=\mathbb{P}_{\crans{q}{T}}$ and $\mathbb{P}_{\kers{p}{T}}=\mathbb{P}_{\kers{q}{T}}$ on $L^{p}\cap L^{q}$, where $\mathbb{P}_{\crans{p}{T}}$ denotes the projection onto $\crans{p}{T}$ along $\kers{p}{T}$ and $\mathbb{P}_{\kers{p}{T}}$ denotes the projection onto $\kers{p}{T}$ along $\crans{p}{T}$.
\end{enumerate}
\end{remark}

\begin{proposition}\label{prop2.7}
For $q\in \mathcal{I}_{D,B}$,  $\mathbb{P}_{\crans{q}{D}}:\crans{q}{BD}\rightarrow \crans{q}{D}$ is an isomorphism with  inverse $\mathbb{P}_{\crans{q}{BD}}:\crans{q}{D}\rightarrow \crans{q}{BD}$.
\end{proposition}
\begin{proof}
Let $h\in \crans{q}{BD}$. Then $h-\mathbb{P}_{\crans{q}{D}}h\in \kers{q}{BD}=\kers{q}{D}$ according to (\ref{EQ.2.0}) for $D$ and \cite[Proposition 2.1 (3)]{Aus-Sta-1}. Thus $\mathbb{P}_{\crans{q}{BD}}\big(h-\mathbb{P}_{\crans{q}{D}}h\big)=0$ and we see that $\mathbb{P}_{\crans{q}{BD}}:\crans{q}{D}\rightarrow \crans{q}{BD}$ is the left inverse of $\mathbb{P}_{\crans{q}{D}}:\crans{q}{BD}\rightarrow \crans{q}{D}$. Reversing the roles of $D$ and $BD$  shows that $\mathbb{P}_{\crans{q}{BD}}:\crans{q}{D}\rightarrow \crans{q}{BD}$ is the right inverse of $\mathbb{P}_{\crans{q}{D}}:\crans{q}{BD}\rightarrow \crans{q}{D}$.
\end{proof}

\begin{remark}[Similarity Property]\label{rem2.8}
For $q\in \mathcal{I}_{D,B}$ we know that $B:\crans{q}{D}\rightarrow \crans{q}{BD}$ is an isomorphism by \cite[Proposition 2.1 item (2)]{Aus-Sta-1}. In particular, for $f\in \Hinf{\nu}$ we have $f\left(DB\right)=B^{-1}f\left(BD\right)B$ on $\crans{q}{D}$ and $f\left(BD\right)=Bf\left(DB\right)B^{-1}$ on $\crans{q}{BD}$.
\end{remark}

\begin{remark}[The interval for the adjoint operators]\label{remark2.6}
Let $\mathcal{A}':=\left\{q/\left(q-1\right)|q\in \mathcal{A}\right\}$ for a subset $\mathcal{A}\subset\left(1,\infty\right)$. By Theorem \ref{theorem2.5} and \cite[Corollary 2.6]{Aus-Sta-1} we have 
\begin{align*}
\left(\mathcal{I}_{D,B}\right)'=\mathcal{I}_{D^{*},B^{*}}.
\end{align*}
\end{remark}

\begin{remark}[The interval for $B=Id$]
Assumption \ref{Assumption2.4} and Proposition \ref{proposition2.9} imply $\mathcal{I}_{D,Id}=\left(1,\infty\right)$ and $\mathcal{I}_{D,B}\subset \mathcal{I}_{D,Id}$. Thus, whenever we are allowed to use the conclusion of Theorem \ref{theorem2.5} for $T\in \left\{DB,BD\right\}$, we are also allowed to use the conclusion of Theorem \ref{theorem2.5} for $D$.
\end{remark}

Under Assumption \ref{Assumption2.4} we are allowed to use kernel/range decompositions and it will be helpful in the sequel to have some properties for the range and the null space. These observations were made in \cite{Hyt-McI}.

\begin{lemma}\cite[Section 3.3]{Hyt-McI}\label{lemma2.8}
For $p,q\in \mathcal{I}_{D,B}$ and $T\in\left\{BD,DB\right\}$ the following statements are true:
\begin{enumerate}
\item \label{Enu1lem2.8} We have with respect to $L^{p}$-topology the direct sum decomposition 
\begin{align*}
L^{p}\cap L^{q}=\left[\kers{p}{T}\cap \kers{q}{T}\right]\oplus\left[\crans{p}{T}\cap \crans{q}{T}\right]
\end{align*}
\item \label{Enu2lem2.8} $\kers{p}{T}\cap \kers{q}{T}$ is dense in $\kers{p}{T}$ with respect to $L^{p}$-topology.
\item \label{Enu3lem2.8} $\rans{p}{T}\cap \rans{q}{T}$ is dense in $\crans{p}{T}$ with respect to $L^{p}$-topology.
\item \label{Enu5lem2.8} $\kers{p}{T}\cap L^{q}\subset \kers{q}{T}$. Hence $\kers{p}{T}\cap \crans{q}{T}=\left\{0\right\}$.
\end{enumerate}
\end{lemma}

\section{The $L^{p}-L^{q}$ Theory}
\subsection{$L^{p}-L^{q}$ estimates in terms of decay properties of holomorphic functions}
There are connections of the bounded holomorphic functional calculus on $L^{q}$ and $L^{q}-L^{q}$ off-diagonal estimates as described in the next lemma, which comes essentially from \cite[Lemma 3.6]{Aus-McI-Rus}. Compare also \cite[Lemma 7.3]{Hyt-Nee-Por} and \cite[Lemma 2.28]{Hof-May-McI}. Before, let us define for $f\in \Hinfzer{\nu}$ the function $f_{t}\in \Hinfzer{\nu}$ by $f_{t}\left(\lambda\right):=f\left(t\lambda\right)$, $\lambda\in\opbisec{\nu}\cup\left\{0\right\}$.
So, one can define families of bounded operators $\left\{f_{t}\left(T\right)\right\}_{t>0}$ via the family of functions $\left\{f_{t}\right\}_{t>0}$. 
\begin{proposition}[$L^{q}-L^{q}$ off-diagonal estimates]\label{proposition3.1}
Let $T\in\left\{DB,BD\right\}$. Denote by $\omega=\omega_{DB}=\omega_{BD}$ the type of bisectoriality and let $\nu\in\left(\omega,\frac{\pi}{2}\right)$. Let $\sigma>0$ and $\tau>0$ be positive real numbers and $q\in \mathcal{I}_{D,B}=\left(p_{-}\left(D,B\right),p_{+}\left(D,B\right)\right)$. Suppose that $\psi\in\Psidecpar{\sigma}{\tau}{\nu}$ and $g\in \Hinf{\nu}$. Then the family of operators  $\left\{g\left(T\right)\psi_{t}\left(T\right)\right\}_{t>0}$ satisfies $L^{q}-L^{q}$ off-diagonal estimates of order $\sigma$.
\end{proposition}

This is interesting in view of the following example.

\begin{example}[Off-diagonal estimates and the semigroup]\label{example3.2}
Here, let us denote  $\sgn z:= \sgn(\Re z)$,  $\widetilde z :=\sgn\left(z\right)z$, $z\in \opbisec{\nu}$ and $|T|:=\sgn\left(T\right)T$.
\begin{enumerate}
\item Let $T\in\left\{DB,BD\right\}$. Denote by $\omega=\omega_{DB}=\omega_{BD}$ the type of bisectoriality and let $\nu\in\left(\omega,\frac{\pi}{2}\right)$. Then the decomposition $e^{-t|T|}=\big(e^{-t|T|}-(Id+itT)^{-1}\big)+(Id+itT)^{-1}$ shows that the semigroup $\left\{e^{-t|T|}\right\}_{t>0}$ satisfies $L^{q}-L^{q}$ off-diagonal estimates of any order $K\in \left[0,1\right]$ as the function $\psi\left(z\right):=e^{-|z|}-\left(1+iz\right)^{-1},\ z\in \opbisec{\nu}$, satisfies $\psi\in \Psidecpar{1}{1}{\nu}$. 
\item The example $n=N=1$ and $D=-i\frac{d}{dx}$ shows that we can not gain more in general. The kernel of the semigroup $\left\{e^{-t|-i\frac{d}{dx}|}\right\}$ is the Poisson kernel $p_{t}\left(x\right):=\frac{1}{\pi} \frac{t}{|x|^{2}+t^{2}}$.
Thus the semigroup does not satisfy $L^{2}-L^{2}$ off-diagonal estimates of order $K>1$ in general.
\end{enumerate}
\end{example}

For certain operators in the functional calculus of $DB$ (or $BD$ resp.) we obtain even $L^{p}-L^{q}$ off-diagonal estimates and $L^{p}-L^{q}$ boundedness. More precisely we have

\begin{proposition}[$L^{p}-L^{q}$ off-diagonal estimates]\label{proposition3.3}
Let $T\in\left\{DB,BD\right\}$. Denote by $\omega=\omega_{DB}=\omega_{BD}$ the type of bisectoriality and let $\nu\in\left(\omega,\frac{\pi}{2}\right)$.
\newline
Suppose $p,q\in \mathcal{I}_{D,B}$ such that $p<q$ and $\tau>\frac{n}{p}-\frac{n}{q}$. Then there exists $c:=c_{p,q}>0$ such that for all $0\leq K<\frac{\sigma}{c}$ one has: For all $\psi\in\Psidecpar{\sigma}{\tau}{\nu}$ and all $g\in \Hinf{\nu}$ the family $\left\{g\left(T\right)\psi_{t}\left(T\right)\right\}_{t>0}$ satisfies $L^{p}-L^{q}$ off-diagonal estimates of order $K$. Moreover, one can choose 
\begin{align}
c_{p,q}=\left(1-\left(\frac{1}{p}-\frac{1}{q}\right)\left(\frac{1}{p_{-}\left(D,B\right)}-\frac{1}{p_{+}\left(D,B\right)}\right)^{-1}\right)^{-1}.
\end{align}
\end{proposition}
\begin{proof}
First we prove the following claim for $L^{p}-L^{q}$-boundedness, which is a special case of the lemma taking $E=F=\mathbb{R}^{n}$ and $K=0$.
\begin{claim}\label{claim3.4}
Suppose $p,q\in \mathcal{I}_{D,B}$ such that $p<q$. Let  $\psi\in\Psidecpar{\sigma}{\tau}{\nu}$, where $\sigma>0$ and $\tau>\frac{n}{p}-\frac{n}{q}$, and $g\in \Hinf{\nu}$. Then the family $\left\{g\left(T\right)\psi_{t}\left(T\right)\right\}_{t>0}$ is $L^{p}-L^{q}$-bounded.
\end{claim}
The proof of Claim \ref{claim3.4} is organized in several steps. The first step is
\begin{claim}\label{claim3.5}
Suppose $q\in \mathcal{I}_{D,B}$ and $p\in \left[q_{*},q\right]\cap \mathcal{I}_{D,B}$, where the lower Sobolev exponent is defined by $q_{*}=\frac{qn}{q+n}$.
Then for all $\lambda\in \mathbb{C}\backslash\bisec{\nu}$, the operator $\left(Id+\lambda DB\right)^{-1}$ is bounded  from $\crans{p}{D}$ to $\crans{q}{D}$ with
\begin{align*}
||\left(Id+\lambda DB\right)^{-1}u||_{L^{q}}\lesssim |\lambda|^{\frac{n}{q}-\frac{n}{p}}||u||_{L^{p}}.
\end{align*}
\end{claim}
\begin{proof}[Proof of Claim \ref{claim3.5}]
We first consider estimates for the resolvent of $BD$ and use the similarity property to pass to $DB$ later on. As 
$$\left(Id+iBD\right)^{-1}:\crans{p}{BD}\rightarrow \crans{p}{BD}$$ and 
$$\left(Id+iBD\right)^{-1}:\crans{p}{BD}\rightarrow \doms{p}{BD}$$ we deduce 
$$\proj{\crans{p}{D}}\left(Id+iBD\right)^{-1}:\crans{p}{D}\rightarrow \crans{p}{D}\cap\doms{p}{D}\subset W^{1,p}.$$
Thus, by Sobolev embedding theorem, we obtain $\proj{\crans{p}{D}}\left(Id+iBD\right)^{-1}u\in L^{q}$ for all $u\in \crans{p}{BD}$ with
$$||\proj{\crans{p}{D}}\left(Id+iBD\right)^{-1}||_{L^{q}}\leq C||u||_{L^{p}}.$$
Now if, moreover, $u\in \crans{q}{BD}$ then $\left(Id+iBD\right)^{-1}u\in \crans{q}{BD}$. By the variant of Remark \ref{rem2.6} for $D$ we have $\proj{\crans{p}{D}}=\proj{\crans{q}{D}}$ on $L^{p}\cap L^{q}$. From that we deduce 
$$\proj{\crans{p}{D}}\left(Id+iBD\right)^{-1}u=\proj{\crans{q}{D}}\left(Id+iBD\right)^{-1}u\in \crans{p}{D}\cap\crans{q}{D}$$
for all $u\in \crans{p}{BD}\cap\crans{q}{BD}$. Since $\proj{\crans{q}{D}}:\crans{q}{BD}\rightarrow\crans{q}{D}$ is an isomorphism by Proposition \ref{prop2.7}, we get 
$$||\left(Id+iBD\right)^{-1}u||_{L^{q}}\lesssim ||\proj{\crans{q}{D}}\left(Id+iBD\right)^{-1}u||_{L^{q}}\lesssim ||u||_{L^{p}}$$
for all $u\in \crans{p}{BD}\cap\crans{q}{BD}$.
By Remark \ref{rem2.8} we know that $B:\crans{p}{D}\rightarrow \crans{p}{BD}$ and $B:\crans{q}{D}\rightarrow \crans{q}{BD}$ are isomorphisms. Thus the similarity property in Remark \ref{rem2.8} yields 
$$||\left(Id+iDB\right)^{-1}u||_{L^{q}}\lesssim ||u||_{L^{p}}$$
for all $u\in \crans{p}{BD}\cap\crans{q}{BD}$.
Now, we use a rescaling argument and note that for $\lambda\in \mathbb{C}\backslash \overline{S}_{\nu}$, $B_{\lambda}$ defined by multiplication of $B_{\lambda}\left(x\right):=-ie^{i\arg \lambda} B\left(|\lambda| x\right)$ has the same properties as $B$ with uniform bounds in $\arg\lambda$. Let $u_{\lambda}\left(x\right):=u\left(|\lambda| x\right)$. Then we have as above 
\begin{align*}
||\left(Id+iDB_{\lambda}\right)^{-1}u_{\lambda}||_{L^{q}}
&\lesssim ||u_{\lambda}||_{L^{p}}
\end{align*} 
and substitution $|\lambda| x\mapsto x$ yields the estimate
\begin{align*}
||\left(Id+\lambda DB\right)^{-1}u||_{L^{q}}
&\lesssim |\lambda|^{\frac{n}{q}-\frac{n}{p}}||u||_{L^{p}}.
\end{align*}
for all $u\in \crans{p}{D}\cap\crans{q}{D}$. By density, the operator $\left(Id+\lambda DB\right)^{-1}$ has the desired extension to $\crans{p}{D}$.
\end{proof}
The second step is
\begin{claim}\label{claim3.6}
Suppose $q\in \mathcal{I}_{D,B}$ and $p\in \left[q_{*},q\right]\cap \mathcal{I}_{D,B}$. Let  $\psi\in\Psidecpar{\sigma}{\tau}{\nu}$, where $\sigma>0$ and $\tau>\frac{n}{p}-\frac{n}{q}$, and $g\in \Hinf{\nu}$. Then we have
\begin{align*}
||g\left(DB\right)\psi_{t}\left(DB\right)u||_{L^{q}}
\lesssim t^{\frac{n}{q}-\frac{n}{p}}||u||_{L^{p}}
\end{align*}
for all $t>0$ and all $u\in L^{p}\cap L^{q}$ (By density even for all $u\in L^{p}$).
\end{claim}
\begin{proof}[Proof of Claim \ref{claim3.6}]
If $u\in \crans{p}{D}\cap \crans{q}{D}$ we have for each $\theta\in \left(\omega,\nu\right)$
\begin{align}
||g\left(DB\right)\psi_{t}\left(DB\right)u||_{L^{q}}
&\lesssim \int_{\partial \bisec{\theta}} |g\left(\lambda\right)||\psi\left(t\lambda\right)|||\left(Id-\lambda^{-1}DB\right)^{-1}u||_{L^{q}}\left|\frac{d\lambda}{\lambda}\right|\nonumber\\
&\lesssim t^{\frac{n}{q}-\frac{n}{p}} \int_{\partial \bisec{\theta}} |\psi\left(t\lambda\right)||t\lambda|^{\frac{n}{p}-\frac{n}{q}}\left|\frac{d\lambda}{\lambda}\right| ||u||_{L^{p}} \nonumber\\
&\lesssim t^{\frac{n}{q}-\frac{n}{p}} ||u||_{L^{p}} \label{EQ.3.2}
\end{align}
by Claim \ref{claim3.5}, Definition \ref{EQ.2.3} and the decay properties of $\psi$. For $u\in L^{p}\cap L^{q}$ we can use the the decomposition in Lemma \ref{lemma2.8}, (\ref{Enu1lem2.8}), associated to the operator $DB$ and $\psi_{t}\left(DB\right)\widetilde{u}=0$ for all $\widetilde{u}\in \kers{p}{DB}\cap \kers{q}{DB}$.
\end{proof}
The third step is the proof of Claim \ref{claim3.4} in the case $T=DB$. 
\newline
Let us denote $q_{0}:=q$ and $q_{l}:=\left(q_{l-1}\right)_{*}$ for $l\in \mathbb{N}^{*}$ and $k:=\inf\left\{l\in \mathbb{N}^{*}: q_{l}\leq p\right\}$. Further, we set
\begin{align*}
\delta&:=\frac{1}{k+1}\left(\tau -\left(\frac{n}{p}-\frac{n}{q}\right)\right)\\
m_{l}&:=1+\delta =\frac{n}{q_{l}}-\frac{n}{q_{l-1}}+\delta &\text{for } 1\leq l <k \text{ and }\\ m_{k}&:=\frac{n}{p}-\frac{n}{q_{k-1}}+\delta=\frac{n}{p}-\frac{n}{q_{k-1}}+\delta. 
\end{align*}
Then we factorize 
\begin{align*}
\psi\left(z\right)= \left(\prod_{l=1}^{k} \left(\frac{1+\widetilde{z}}{1+\widetilde{z}}\right)^{m_{l}}\right)\cdot\left(\frac{1+\widetilde{z}}{\widetilde{z}}\right)^{\frac{k\sigma}{k+1}}\cdot\left(\frac{\widetilde{z}}{1+\widetilde{z}}\right)^{\frac{k\sigma}{k+1}}\cdot\psi\left(z\right)=:\zeta\left(z\right)\prod_{l=1}^{k} \xi^{l}\left(z\right)
\end{align*}
where $\widetilde{z}:=\sgn\left(z\right)z$ and 
\begin{align*} \zeta\left(z\right)&:=\left(\prod_{l=1}^{k}\left(1+\widetilde{z}\right)^{m_{l}}\right)\left(\frac{1+\widetilde{z}}{\widetilde{z}}\right)^{\frac{k\sigma}{k+1}}\psi\left(z\right), \\ \xi^{l}\left(z\right)&:=\left(\frac{\widetilde{z}}{1+\widetilde{z}}\right)^{\frac{\sigma}{k+1}}\left(\frac{1}{1+\widetilde{z}}\right)^{m_{l}}.
\end{align*} 
We observe that each $\xi^{l}$ satisfies the conditions of Claim \ref{claim3.6}:
\newline
$\xi^{l}\in \Psidecpar{\sigma_{l}}{\tau_{l}}{\nu}$ where $\sigma_{l}>0$, $\tau_{l}>\frac{n}{q_{l}}-\frac{n}{q_{l-1}}$, $\sigma_{k}>0$, $\tau_{k}>\frac{n}{p}-\frac{n}{q_{k-1}}$ and $\zeta\in \Psidec{\nu}$.
\newline
Hence, we have 
\begin{align}
&\xi^{l}_{t}\left(DB\right): L^{q_{l-1}}\rightarrow L^{q_{l}},\\
&\xi^{k}_{t}\left(DB\right): L^{q_{k-1}}\rightarrow L^{p},\\
&\zeta_{t}\left(DB\right): L^{p}\rightarrow L^{p}.
\end{align}
Now, Claim \ref{claim3.4} in the case $T=DB$ follows by iteration of Claim \ref{claim3.6}.

\

The fourth step is to deduce Claim \ref{claim3.4} in the case $T=BD$.
From the case $T=DB$ just proved, the similarity property $g\left(BD\right)\psi_{t}\left(BD\right)=Bg\left(DB\right)\psi_{t}\left(DB\right)B^{-1}$ on $\crans{q}{BD}$ and $\crans{p}{BD}$, the boundedness and coercivity of $B$ on $\crans{q}{D}$ and $\crans{p}{D}$ we get 
\begin{align*}
||g\left(BD\right)\psi_{t}\left(BD\right)u||_{L^{q}}
\lesssim t^{\frac{n}{q}-\frac{n}{p}}||B^{-1}u||_{L^{p}}
\lesssim t^{\frac{n}{q}-\frac{n}{p}}||u||_{L^{p}}
\end{align*}
for all $u\in \crans{q}{BD}\cap\crans{p}{BD}$.
In the general case $u\in L^{q}\cap L^{p}$ we can use the decomposition in Lemma \ref{lemma2.8}, (\ref{Enu1lem2.8}) associated to the operator $BD$ and $\psi_{t}\left(BD\right)\widetilde{u}=0$ for all $\widetilde{u}\in \kers{p}{BD}\cap \kers{q}{BD}$. By density we conclude the assertion 
\begin{align*}
||g\left(BD\right)\psi_{t}\left(BD\right)u||_{L^{q}}
&\lesssim t^{\frac{n}{q}-\frac{n}{p}}||u||_{L^{p}}
\end{align*}
for all $u\in L^{p}$. So, Claim \ref{claim3.4} is completely proved.

\

Now, we turn to the conclusion of Proposition \ref{proposition3.3} using Claim \ref{claim3.4}. 
First 
by normalizing,  we may  assume $||\psi||_{\Hinf{\nu}}=||g||_{\Hinf{\nu}}=1$. We combine $L^{r}-L^{r}$ off-diagonal estimates and $L^{p_{0}}-L^{q_{0}}$ boundedness to conclude $L^{p}-L^{q}$ off-diagonal estimates by interpolation, where $p,q,r,p_{0},q_{0}\in \mathcal{I}_{D,B}$. Since we use $L^{p_{0}}-L^{q_{0}}$ boundedness we have to make sure that the family of holomorphic functions has enough decay at infinity to use $L^{p_{0}}-L^{q_{0}}$ boundedness. So, we define $\zeta^{\alpha}_{t}\left(z\right):=g\left(z\right)\left(1+t\widetilde{z}\right)^{\alpha}\psi_{t}\left(z\right)$, where we recall $\widetilde{z}=\sgn\left(\Re z\right)z$ for $z\in \opbisec{\nu}$ and $\alpha\in \mathbb{C}$ such that $\Re \alpha<\tau$ and observe that for fixed $t>0$ the operator $g\left(T\right)\psi_{t}\left(T\right)$ is embedded in the analytic family $\left\{\zeta^{\alpha}_{t}\left(T\right)\right\}_{\alpha}$. Polar coordinates and $\arg\left(1+t\widetilde{z}\right)\in \left(-\nu,\nu\right)$ yields
\begin{align*}
\sup_{z\in \opbisec{\nu}}\left|\left(1+t\widetilde{z}\right)^{\alpha}\right|\leq e^{\nu |\Im \alpha|}|tz|^{\Re \alpha}.
\end{align*}
Using polar coordinates we can calculate that 
$$|\zeta^{\alpha}_{t}\left(z\right)|\lesssim e^{\nu |\Im \alpha|}\inf\left\{|tz|^{\sigma},|tz|^{\Re \alpha-\tau}\right\}$$ and consequently the symbol satisfies $||\zeta^{\alpha}_{t}||_{\Hinf{\nu}}\lesssim 1$. Thus we can deduce from Proposition \ref{proposition3.1}
\begin{align*}
||\chi_{F}\zeta^{\alpha}_{t}\left(T\right)\left(\chi_{E}u\right)||_{L^{r}} \lesssim e^{\nu |\Im \alpha|} \left(1+\frac{d\left(E,F\right)}{t}\right)^{-\sigma}||\chi_{E}u||_{L^{r}}
\end{align*}
for all $r\in \mathcal{I}_{D,B}$ and all $\alpha\in \mathbb{C}$ such that $\tau-\Re \alpha>\left(\frac{n}{p}-\frac{n}{q}\right)-\Re \alpha>0$. Now, let $p_{0},q_{0}\in \mathcal{I}_{D,B}$. We have for all $\tau-\Re \alpha>\left(\frac{n}{p}-\frac{n}{q}\right)-\Re \alpha>\left(\frac{n}{p_{0}}-\frac{n}{q_{0}}\right)$ 
\begin{align*}
||\chi_{F}\zeta^{\alpha}_{t}\left(T\right)\left(\chi_{E}u\right)||_{L^{q_{0}}} \lesssim e^{\nu |\Im \alpha|} t^{\frac{n}{q_{0}}-\frac{n}{p_{0}}}||\chi_{E}u||_{L^{p_{0}}}.
\end{align*}
by Claim \ref{claim3.4}. Next, we will use Stein's interpolation theorem for the analytic family of operators $\left\{\zeta^{\alpha}_{t}\left(T\right)\right\}_{\alpha}$ with 
\begin{align}\label{EQ.3.6}
&\frac{1}{p}=\frac{1-\theta}{r}+\frac{\theta}{p_{0}} && \frac{1}{q}=\frac{1-\theta}{r}+\frac{\theta}{q_{0}}.
\end{align}
and $\theta:=\left(\frac{n}{p}-\frac{n}{q}-\Re \alpha\right)\left(\frac{n}{p_{0}}-\frac{n}{q_{0}}\right)^{-1}$ at $\Re \alpha=0$. This yields 
\begin{align*}
||\chi_{F}\zeta^{\alpha}_{t}\left(T\right)\left(\chi_{E}u\right)||_{L^{q}} \lesssim M_{\Im \alpha} t^{c_{1}\left(\frac{n}{q_{0}}-\frac{n}{p_{0}}\right)} \left(1+\frac{d\left(E,F\right)}{t}\right)^{-c_{0}\sigma}||\chi_{E}u||_{L^{p}}
\end{align*}
when $\Re \alpha=0$. The constants $c_{0},c_{1}$ are related to the formula in \cite[Theorem  1.3.7]{Gra}. Choosing $\alpha=0$ yields 
\begin{align}\label{EQ.3.7}
||\chi_{F}g\left(T\right)\psi_{t}\left(T\right)\left(\chi_{E}u\right)||_{L^{q}} \lesssim  t^{c_{1}\left(\frac{n}{q_{0}}-\frac{n}{p_{0}}\right)} \left(1+\frac{d\left(E,F\right)}{t}\right)^{-c_{0}\sigma}||\chi_{E}u||_{L^{p}}
\end{align}
By \cite[Theorem 1.3.7, Exercise 1.3.8]{Gra} we know that 
\begin{align*}
c_{0}&=1-\theta
=1-\left(\frac{1}{p}-\frac{1}{q}\right)\left(\frac{1}{p_{0}}-\frac{1}{q_{0}}\right)^{-1}\\
c_{1}&=\theta
=\left(\frac{1}{p}-\frac{1}{q}\right)\left(\frac{1}{p_{0}}-\frac{1}{q_{0}}\right)^{-1}
\end{align*}
Thus, (\ref{EQ.3.7}) reads 
\begin{align}\label{EQ.3.8}
||\chi_{F}g\left(T\right)\psi_{t}\left(T\right)\left(\chi_{E}u\right)||_{L^{q}} 
\lesssim  t^{\frac{n}{q}-\frac{n}{p}} \left(1+\frac{d\left(E,F\right)}{t}\right)^{-\left(1-\theta\right)\sigma}||\chi_{E}u||_{L^{p}}.
\end{align}
Since $p,q$ are fixed in the relation and $r$ is choosen depending on $p_{0},q_{0},$ the parameter $\theta=\theta\left(p_{0},q_{0}\right)$ is determined by $p_{0},q_{0}$. In order to minimize the factor $\left(1+\frac{d\left(E,F\right)}{t}\right)^{-\sigma\left(1-\theta\right)}$ in (\ref{EQ.3.8}), we minimize $\theta$ using (\ref{EQ.3.6}). Indeed, we get by (\ref{EQ.3.6}) the relation $\theta=\theta\left(p_{0},q_{0}\right):=\left(\frac{1}{p}-\frac{1}{q}\right)\left(\frac{1}{p_{0}}-\frac{1}{q_{0}}\right)^{-1}$ and observe that
\begin{align*} \inf\left\{\theta\left(p_{0},q_{0}\right)|p_{0},q_{0}\in \mathcal{I}_{D,B}\right\}=\left(\frac{1}{p}-\frac{1}{q}\right)\left(\frac{1}{p_{-}\left(D,B\right)}-\frac{1}{p_{+}\left(D,B\right)}\right)^{-1}.
\end{align*}
As we are allowed to choose $p_{-}\left(D,B\right)<p_{0}<q_{0}< p_{+}\left(D,B\right)$ arbritrary in (\ref{EQ.3.6}) we get the estimate 
\begin{align*}
||\chi_{F}g\left(T\right)\psi_{t}\left(T\right)\left(\chi_{E}u\right)||_{L^{q}} 
&\lesssim t^{\frac{n}{q}-\frac{n}{p}}\left(1+\frac{d\left(E,F\right)}{t}\right)^{-K} ||\chi_{E}u||_{L^{p}}.
\end{align*}
for each $K\in \left[0,\infty\right)$ such that $\sigma>Kc_{p,q}$, where
\begin{align*}
c_{p,q}:=\left(1-\left(\frac{1}{p}-\frac{1}{q}\right)\left(\frac{1}{p_{-}\left(D,B\right)}-\frac{1}{p_{+}\left(D,B\right)}\right)^{-1}\right)^{-1}.
\end{align*}
\end{proof}

The next example shows that there are families of operators with finite $\sigma$ in the functional calculus, which satisfy off-diagonal estimates of each order $K\in \left[0,\infty\right)$, showing that the condition $\sigma>cK$ is sufficient but not necessary.

\begin{example}[$L^{p}-L^{q}$ off-diagonal estimates of arbritrary order]\label{example3.7}
Let $T\in\left\{BD,DB\right\}$ and $\alpha,M\in \mathbb{N}^{*}$ with $0<\alpha\leq M$. Then the family $\left\{\left(itT\right)^{\alpha}\left(Id+itT\right)^{-M}\right\}_{t>0}$ satisfies $L^{p}-L^{q}$ off-diagonal estimates of order $K$ for each $K\in \left[0,\infty\right)$ whenever $p,q\in \mathcal{I}_{D,B}$ with $p<q$ such that $M-\alpha>\frac{n}{p}-\frac{n}{q}$.
\end{example}

\begin{remark}
We do not know if the condition $\tau>\frac{n}{p}-\frac{n}{q}$ is necessary in Proposition \ref{proposition3.3}.
\end{remark}

Sometimes, it is appropriate to have the following variant of Proposition \ref{proposition3.3}. See for example \cite{Aus-Sta-2}.
\begin{proposition}[$L^{p}-L^{q}$ biparameter off-diagonal estimates]\label{proposition3.8}
Let $T\in \left\{DB,BD\right\}$. Denote by $\omega:=\omega_{DB}=\omega_{BD}$ the type of bisectoriality and let $\nu\in \left(\omega,\frac{\pi}{2}\right)$. 
\newline
Suppose $p,q\in \mathcal{I}_{D,B}$ such that $p<q$ and let $\sigma>0$, $\tau>\frac{n}{p}-\frac{n}{q}$. Then there exists $c:=c_{p,q}>0$ such that for $0\leq K<\frac{M}{c}$ one has: Suppose that $\psi\in\Psidecpar{\sigma}{\tau}{\nu}$ and $\varphi\in \Hinf{\nu}$ are functions such that $\varphi$ satisfies $|\varphi\left(\lambda\right)|\lesssim \inf \left\{|\lambda|^{M},1\right\}$ for all $\lambda\in \opbisec{\nu}$. Then the family $\left\{\psi_{t}\left(T\right)\varphi_{r}\left(T\right)\right\}_{t\geq r>0}$ satisfies $L^{p}-L^{q}$ biparameter off-diagonal estimates in $\left(t,r\right)$ of order $K$. Moreover, one can choose 
\begin{align}
c_{p,q}=\left(1-\left(\frac{1}{p}-\frac{1}{q}\right)\left(\frac{1}{p_{-}\left(D,B\right)}-\frac{1}{p_{+}\left(D,B\right)}\right)^{-1}\right)^{-1}.
\end{align}
\end{proposition}

\begin{proof}
The conclusion of Proposition \ref{proposition3.8} follows by analytic interpolation as in Proposition \ref{proposition3.3}: in fact, we use interpolation between Claim \ref{claim3.4} (that is Proposition \ref{proposition3.3} in the case $K=0$.) and the next claim.
\begin{claim}\label{claim3.9}
With the assumption above and $p=q$ the family $\left\{\psi_{t}\left(T\right)\varphi_{r}\left(T\right)\right\}_{t\geq r>0}$ satisfies $L^{q}-L^{q}$ biparameter off-diagonal estimates in $\left(t,r\right)$ of order $M$.
\end{claim}
\begin{proof}[Proof of Claim \ref{claim3.9}]
W.l.o.g. assume $||\psi||_{\Hinf{\nu}}\leq 1$. Let $u\in L^{q}$ with $\supp u\subset E$. We have for each $\theta\in \left(\omega,\nu\right)$
\begin{align}
||\psi_{t}\left(T\right)\varphi_{r}\left(T\right)u||_{L^{q}\left(F\right)}&\lesssim \int_{\partial \bisec{\theta}} |\varphi\left(r\lambda\right)||\psi\left(t\lambda\right)|||\left(Id-\lambda^{-1}T\right)^{-1}u||_{L^{q}\left(F\right)}\left|\frac{d\lambda}{\lambda}\right| \nonumber\\
&\lesssim \int_{\partial \bisec{\theta}} |\varphi\left(\lambda\right)||\psi\left(\frac{t\lambda}{r}\right)|\left(1+\frac{d\left(E,F\right)}{r}|\lambda|\right)^{-K}\left|\frac{d\lambda}{\lambda}\right| ||u||_{L^{q}\left(E\right)} \nonumber \\
&\lesssim \left(1+\frac{d\left(E,F\right)}{r}\right)^{-M}||u||_{L^{q}\left(E\right)} \label{EQ.3.10}
\end{align}
where $K\in \left[0,\infty\right)$ will be chosen below.
For the proof of (\ref{EQ.3.10}) we consider two cases. On the one hand, if $\frac{d\left(E,F\right)}{r}\leq 1$ we have
\begin{align*}
&\int_{\partial \bisec{\theta}}|\varphi\left(\lambda\right)||\psi\left(\frac{t\lambda}{r}\right)| \left(1+\frac{d\left(E,F\right)}{r}|\lambda|\right)^{-K}\left|\frac{d\lambda}{\lambda}\right|\\
&\leq \int_{\partial \bisec{\theta}}|\varphi\left(\lambda\right)|\cdot \inf\left\{1,|\lambda|^{-\tau}\right\}\left|\frac{d\lambda}{\lambda}\right|\\
&\lesssim 1.
\end{align*}
In fact, the last estimate follows by splitting the contour integral at $|\lambda|=1$ and using that 
\begin{align*}
|\psi\left(\frac{t\lambda}{r}\right)|\lesssim \begin{cases}||\psi||_{\Hinf{\nu}}\leq 1, &\text{if} \ |\lambda|\leq 1,\\
|\frac{t\lambda}{r}|^{-\tau}\leq |\lambda|^{-\tau}, &\text{if} \ |\lambda|\geq 1.
\end{cases}
\end{align*}
On the other hand if $x:=\frac{d\left(E,F\right)}{r}\geq 1$, we split 
\begin{align*}
\int_{\partial \bisec{\theta}}|\varphi\left(\lambda\right)||\psi\left(\frac{t\lambda}{r}\right)| \left(1+\frac{d\left(E,F\right)}{r}|\lambda|\right)^{-K} \left|\frac{d\lambda}{\lambda}\right|
\end{align*}
into three parts according to $|\lambda|\leq \frac{1}{x}$, $\frac{1}{x}\leq |\lambda| \leq 1$ and $|\lambda|\geq 1$. From this the estimate 
\begin{align}\label{eq.3.11}
\int_{\partial \bisec{\theta}}|\varphi\left(\lambda\right)||\psi\left(\frac{t\lambda}{r}\right)| \left(1+\frac{d\left(E,F\right)}{r}|\lambda|\right)^{-K} \left|\frac{d\lambda}{\lambda}\right|
\lesssim \left(\frac{d\left(E,F\right)}{r}\right)^{-M}.
\end{align}
easily follows, required we choose $K>M$. In fact, for the first part we use $\left(1+\frac{d\left(E,F\right)}{r}|\lambda|\right)^{-K}\leq 1$ and for the second and third part we estimate 
$$\left(1+\frac{d\left(E,F\right)}{r}|\lambda|\right)^{-K}\leq \left(\frac{d\left(E,F\right)}{r}\right)^{-K}|\lambda|^{-K}$$ 
for the same choice of $K>M$ and evaluate the three integrals associated to the three parts. The addition of the three evaluated parts is bounded by the right hand side in (\ref{eq.3.11}).
\end{proof}
The lemma is proved.
\end{proof}

For the semigroup $e^{-t|T|}$ for $T\in \left\{BD,DB\right\}$ we can only prove $\crans{p}{T}-\crans{q}{T}$ boundedness whenever $p,q\in\mathcal{I}_{D,B}$ with $p\leq q$. More precisely, we prove that $f_{t}\left(T\right)$ maps $\crans{p}{T}$ to $\crans{q}{T}$, whenever the holomorphic function $f$ has enough decay at infinity. We will apply this result to prove a Hardy-Littlewood-Sobolev inequality for fractional operators $|T|^{-\alpha}$ in the next section.

\begin{corollary}[$L^{p}-L^{q}$ theory for bounded holomorphic functions]\label{corollary3.13}
Let $T\in \left\{DB,BD\right\}$. Denote by $\omega:=\omega_{DB}=\omega_{BD}$ the type of bisectoriality and let $\nu\in \left(\omega,\frac{\pi}{2}\right)$. 
Suppose $p,q\in \mathcal{I}_{D,B}$ such that $p\leq q$, and let $g\in \Hinf{\nu}$ and $f$ be holomorphic function with $|f\left(\lambda\right)|\lesssim \inf\left\{1,|\lambda|^{-M}\right\}$ for all $\lambda\in \opbisec{\nu}$, where $M>\frac{n}{p}-\frac{n}{q}$. Then the family $\left\{g\left(T\right)f_{t}\left(T\right)\right\}_{t>0}$ is $\crans{p}{T}-\crans{q}{T}$ bounded.
\newline
In particular, the semigroup $\left\{e^{-t|T|}\right\}_{t>0}$ is $\crans{p}{T}-\crans{q}{T}$ bounded for all $p,q\in \mathcal{I}_{D,B}$ with $p\leq q$.
\end{corollary}
\begin{proof}
The first part can be proved using McIntosh convergence lemma and ideas from the proof in proposition \ref{proposition3.3}. We let this to the interested reader. The statement for the semigroup follows from the special choice $g=1$ and $f\left(z\right)=e^{-\widetilde{z}}$, where $\widetilde{z}=\sgn\left(\Re z\right)z$ as usual.
\end{proof}

\begin{remark}
In the situation of Corollary \ref{corollary3.13}, the family $\left\{g\left(T\right)f_{t}\left(T\right)\right\}_{t>0}$ is $L^{p}-L^{q}$ bounded whenever $\left(fg\right)\left(0\right)=0$. This follows from Corollary \ref{corollary3.13} and (\ref{EQ.2.4}). We treat the case $\left(fg\right)\left(0\right)\neq0$ in Subsection \ref{subsection3.3}.
\end{remark}

\subsection{Stability under multiplication by cut-off functions and the relation to Ajiev's work}
Let us begin this subsection with definition of stability under multiplication by smooth cut-off functions/cut-off functions.

\begin{definition}
Let $U_{q}$ be a closed subspace of $L^{q}$, $1\leq q <\infty$. 
\begin{itemize}
\item We say $U_{q}$ is stable under multiplication by cut-off functions if for any $u\in U_{q}$ and any characteristic function $\chi_{E}$ to a Borel measurable set $E\subset \mathbb{R}^{n}$ one has $\chi_{E}u\in U_{q}$.
\item We say $U_{q}$ is stable under multiplication by smooth cut-off functions if for any $u\in U_{q}$ and any smooth complex-valued function $\zeta$ with compact support one has $\zeta u\in U_{q}$.
\end{itemize}
\end{definition}

\begin{remark}[Equivalence]
We observe that both notions are equivalent. Indeed, if $U_{q}$ is stable by cut-off functions it is also stable under multiplication by simple functions. Then by an approximation argument and the closedness of $U_{q}$ it follows that $U_{q}$ is stable under multiplication by smooth cut-off functions. Conversely, if  $U_{q}$ is stable under multiplication by smooth cut-off functions, then it follows from the closedness of $U_{q}$ and a mollifier approximation argument that $U_{q}$ is stable under multiplication by cut-off functions.
\end{remark}

\begin{remark}[Relation to Ajiev's work]\label{remark3.16}
\begin{enumerate}
\item\label{Enu1rem3.16} A combination of \cite[Theorem 4.6(a)]{Aji} and \cite[Theorem 4.14]{Aji} imply Corollary \ref{corollary3.13} for a subclass of function pairs $\left(f,g\right)$.
\item\label{Enu2rem3.16} A combination of \cite[Theorem 4.6(b)]{Aji}, \cite[Theorem 4.14]{Aji}, \cite[Lemma 4.13]{Aji} and \cite[Remark 2]{Aji} imply Proposition \ref{proposition3.3} for a subclass of function pairs $\left(\psi,g\right)$, provided that the range $\crans{p}{T}$ is stable under multiplication by (smooth) cut-off functions.
\end{enumerate}
\end{remark}

We never used these notions. However, to compare with \cite{Aji} we investigate whether or not $\crans{p}{D}$ is stable under multiplication by (smooth) cut-off functions.

\begin{definition}
Let $D$ as in Assumption \ref{Assumption2.4} above and $p\in \left(1,\infty\right)$.
\begin{enumerate}
\item We define $V_{p}$ to be the linear subspace of $\mathbb{C}^{N}$ generated by $\dashint_{B} v$ for all balls $B\subset \mathbb{R}^{n}$ and all $v\in \crans{p}{D}$.
\item We define $W_{p'}$ to be the linear subspace of $\mathbb{C}^{N}$ generated by $\dashint_{B} w$ for all balls $B\subset \mathbb{R}^{n}$ and all $w\in \kers{p'}{D^{*}}$
\end{enumerate}
\end{definition}

\begin{remark}\label{Remark 3.17}
By the Lebesgue differentiation theorem any $v\in \crans{p}{D}$ takes values in $V_{p}$ almost everywhere. Similarly, any $w\in \kers{p'}{D^{*}}$ takes values in $W_{p'}$ almost everywhere. Thus, $V_{p}$ is the space of almost everywhere values of all elements in $\crans{p}{D}$ and $W_{p'}$ is the space of values of all elements in $ \kers{p'}{D^{*}}$.
\end{remark}

\begin{remark}
Suppose $1<p,q<\infty$. Then $V_{p}=V_{q}$ and $W_{p'}=W_{q'}$. This follows from the density statements in Lemma \ref{lemma2.8}. Thus, we may set $V=V_{p}$ and $W=W_{p'}$ for one $p\in (1,\infty)$. 
\end{remark}

\begin{proposition}[Stability under multiplication by smooth cut-off functions]\label{proposition..3.18}
Let $D$ as in Assumption \ref{Assumption2.4} above and $p\in \left(1,\infty\right)$.
\begin{enumerate}
\item\label{Enu1prop..3.18} If $\crans{p}{D}=L^{p}$, then $\crans{p}{D}$ is stable under multiplication by smooth cut-off functions.
\item\label{Enu2prop..3.18} If $\crans{p}{D}\neq L^{p}$, then $\crans{p}{D}$ is stable under multiplication by smooth cut-off functions if and only if $V\perp W$ for the $\mathbb{C}^{N}$ inner product.
\end{enumerate}
This implies that (1) or (2) holds for one $p$, it holds for all $p$. 
\end{proposition}

\begin{proof}
Assertion (\ref{Enu1prop..3.18}) is evident, so we turn to the proof of Assertion (\ref{Enu2prop..3.18}).

Since $\crans{p}{D}$ is the polar set to $\kers{p'}{D^{*}}$, we have that $\crans{p}{D}$ is stable by multiplication of smooth cut-off functions if and only if $\left\langle \xi v,w\right\rangle=0$ for all $v\in \crans{p}{D}$, all $w\in \kers{p'}{D^{*}}$ and all smooth cut-off functions $\zeta$. 
We claim that this is equivalent to $ v\overline{w}=0$ almost everywhere for  all $v\in \crans{p}{D}$ and all $w\in \kers{p'}{D^{*}}$. As $v,w$ are arbitrary, this is equivalent to $V\perp W$. We begin with the direction "$\Rightarrow$": If $\left\langle \zeta v,w\right\rangle=0$ for all  such $v,w, \zeta$ then in particular we have   $\left\langle \zeta_{\xi} v,w\right\rangle=0$  for all $\xi \in \mathbb{R}^n$, where $\zeta_{\xi}\left(x\right):=e^{-ix\cdot\xi} \zeta\left(x\right)$. Let us denote by $\mathcal{F}$ the Fourier transform. Then this implies $\mathcal{F}\left(\zeta v\overline{w}\right)\left(\xi\right)=0$ for all $\xi\in \mathbb{R}^n$ by definition of the Fourier transform. As $\zeta v\overline{w}\in L^{1}$, we deduce that $\zeta v\overline{w}=0$ almost everywhere. Choosing all possible $\zeta$, this concludes  the proof of the first direction. We turn to the converse direction. If $ v\overline{w}=0$  almost everywhere for  all $v\in \crans{p}{D}$ and all $w\in \kers{p'}{D^{*}}$ then  $\zeta v\overline{w}=0$ almost everywhere for all these $v,w$ and smooth cut-off $\zeta$, hence $\left\langle \zeta v,w\right\rangle=0$. This shows that $\zeta v$ belongs to the polar set of $\kers{p'}{D^{*}}$, hence 
$\zeta v \in \crans{p}{D}$.  This concludes the proof of the stated equivalence and of  the lemma. 
\end{proof}

\begin{example}
We claim that the spaces $\crans{p}{DB}=\crans{p}{D}$ associated to the operators $DB$ and $D$ in \cite{Aus-Axe-McI-1} and \cite{Aus-Axe}  are not stable under multiplication by smooth cut-off functions. Indeed,  for $\displaystyle D=\begin{pmatrix} 0 & \div \\ -\nabla & 0 \end{pmatrix}$, we have
\begin{align*}
\kers{p'}{D^{*}}=\kers{p'}{D}=\left\{u=\left(0,g\right)\in L^{p'}\left(\mathbb{R}^{n};\mathbb{C}^{m}\oplus[\mathbb{C}^{m}\otimes\mathbb{C}^{n}]\right)\, ;\, 
\div g=0\right\}\ne \{0\},
\end{align*}
hence $W\ne \{0\}$. Next, we have that 
\begin{align*}
\crans{p}{D}=\left\{u=\left(f,g\right)\in L^{p}\left(\mathbb{R}^{n};\mathbb{C}^{m}\oplus[\mathbb{C}^{m}\otimes\mathbb{C}^{n}]\right)\, ;\, 
g=\nabla h, h\in \dot W^{1,p}(\mathbb{R}^{n},\mathbb{C}^{m}) \right\}.
\end{align*}
Let $c\in \mathbb{C}^{m}$ and $\xi\in \mathbb{C}^{m}\otimes\mathbb{C}^{n}= (\mathbb{C}^{m})^n$. Taking  $f\in L^p$ which is constant with value $c$ on some ball and $h\in \dot W^{1,p}$ with 
$h(x)= \sum_{i=1}^n x_{i}\xi_{i}$ in the same ball. We see that $(c, \xi_{1},\ldots, \xi_{n}) \in V$. Thus, $\mathbb{C}^{N}\subset V$ (with $N=m(1+n)$). The claim follows  as we are in the situation of (2) in  Proposition \ref{proposition..3.18} and $W$ is not orthogonal to $V$.

\end{example}

This shows that Ajiev's results do not apply to the main motivating example.

\subsection{$L^{p}-L^{q}$ estimates and the relation to the kernel/range decomposition}\label{subsection3.3}
From the next proposition and example we will learn more about the relation of kernel/range decomposition and the $L^{p}-L^{q}$ boundedness of the related operators in the functional calculus. The proposition shows that $f\left(0\right)=0$ is a necessary condition for functions $f$ to have $L^{p}-L^{q}$ boundedness of the associated operator, whenever the null space is not equal $\left\{0\right\}$. Before we state the proposition we make a definition.

\begin{definition}[Not bounded]
Let $\mathcal{X}$ and $\mathcal{Y}$ be two Banach spaces and $T:\mathcal{X}\rightarrow \mathcal{X}$ be a bounded linear operator. We say $T$ is not bounded from $\mathcal{X}$ to $\mathcal{Y}$ and write $T:\mathcal{X}\nrightarrow \mathcal{Y}$ if there exists $u\in \mathcal{X}$ such that $Tu\notin \mathcal{Y}$ or if there exists no constant $C>0$ such that for all $u\in \mathcal{X}$ holds $||Tu||_{\mathcal{Y}}\leq C ||u||_{\mathcal{X}}$.
\end{definition}

\begin{proposition}[Necessary Condition]\label{proposition3.11}
Let $T\in \left\{DB,BD\right\}$. Denote by $\omega:=\omega_{DB}=\omega_{BD}$ the type of bisectoriality and let $\nu\in \left(\omega,\frac{\pi}{2}\right)$. 
Suppose there exists $r\in \mathcal{I}_{D,B}$ such that $\kers{r}{T}\neq \left\{0\right\}$ and let $f\in \Hinfzer{\nu}$ with $f\left(0\right)\neq0$. Then for all $p,q\in \mathcal{I}_{D,B}$ such that $p\neq q$ we have $f\left(T\right):\kers{p}{T}\nrightarrow L^{q}$. 

In particular, for all $p,q\in \mathcal{I}_{D,B}$ such that $p\neq q$ we have $f\left(T\right):L^{p}\nrightarrow L^{q}$.
\end{proposition}
\begin{proof}
First, we note by Lemma \ref{lemma2.8}, (\ref{Enu2lem2.8}), that $\kers{r}{T}\neq \left\{0\right\}$ for one $r\in \mathcal{I}_{D,B}$ is equivalent to $\kers{r}{T}\neq \left\{0\right\}$ for all $r\in \mathcal{I}_{D,B}$. Thus we can assume $\kers{p}{T}\neq \left\{0\right\}\neq \kers{q}{T}$ for particular $p,q \in \mathcal{I}_{D,B}$. Since $f\left(T\right)u=f\left(0\right)u$ for all $u\in\kers{p}{T}$ we observe that 
the statement $f\left(T\right):\kers{p}{T}\nrightarrow \LRC{q}{n}{N}$ is equivalent to $f\left(T\right):\kers{p}{T}\nrightarrow \kers{q}{T}$, which we prove next.

We begin with the case $T=DB$ and let $p,q\in \mathcal{I}_{D,B}$. Since $\left(\kers{p}{DB}\right)^{*}=\kers{p'}{B^{*}D^{*}}$ and similarly $\left(\kers{q}{DB}\right)^{*}=\kers{q'}{B^{*}D^{*}}$ we observe that $f\left(DB\right):\kers{p}{DB}\rightarrow \kers{q}{DB}$ is equivalent to $f^{*}\left(B^{*}D^{*}\right):\kers{q'}{B^{*}D^{*}}\rightarrow \kers{p'}{B^{*}D^{*}}$ by duality, where $f^{*}\left(\lambda\right):=\overline{f\left(\overline{\lambda}\right)}$ for $\lambda\in \opbisec{\nu}\cup\left\{0\right\}$. Now, recall also that $p,q\in \mathcal{I}_{D,B}$ is equivalent to $p',q'\in \mathcal{I}_{D^{*},B^{*}}$ by Remark \ref{remark2.6}. Thus it suffices to consider the case $T=BD$.

We turn to the case $T=BD$ and assume that for $p,q\in \mathcal{I}_{D,B}$ the operator $f\left(BD\right)$ defined by the bounded holomorphic functional calculus maps $\kers{p}{BD}$ to $\kers{q}{BD}$, with quantitative estimate
\begin{align*}
&||f\left(BD\right)u||_{L^{q}}\leq C ||u||_{L^{p}}, &&\forall u\in \kers{p}{BD}
\end{align*}
where $C$ is of course independent of $u$. Since $f\left(BD\right)u=f\left(0\right)u$ this estimate turns into
\begin{align}\label{EQ.3.11}
&|f\left(0\right)|||u||_{L^{q}}\leq C ||u||_{L^{p}}, &&\forall u\in \kers{p}{BD}.
\end{align}
Since $\kers{p}{BD}=\kers{p}{D}$ is the null space of the constant coefficient partial differential operator we observe that $u\in \kers{p}{BD}$ is equivalent to $u_{s}\in \kers{p}{BD}$ for all $s >0$ by chain rule, where $u_{s}\left(x\right):=u\left(s x\right)$ for all $s>0$ and all $x\in \mathbb{R}^{n}$. This means the null space $\kers{p}{BD}$ is invariant by rescaling. Thus, if we fix $u\in\kers{p}{BD}$ such that $u\neq0$ we get the inequality 
\begin{align*}
&|f\left(0\right)|||u_{s}||_{L^{q}}\leq C ||u_{s}||_{L^{p}} &&\forall s
\end{align*}
from (\ref{EQ.3.11}) above. But by substitution, this inequality is equivalent to the inequality
\begin{align}\label{EQ.3.12}
&s^{\frac{n}{q}-\frac{n}{p}}\leq  \frac{C||u||_{L^{p}}}{|f\left(0\right)|||u||_{L^{q}}}, &&\forall s
\end{align}
for our fixed $u\in\kers{p}{BD}$ with $u\neq 0$. If $p<q$ we get a contradiction in (\ref{EQ.3.12}) as $s\rightarrow 0$. If $p>q$ we get a contradiction in (\ref{EQ.3.12}) as $s\rightarrow \infty$.
\end{proof}

\begin{example}
In the proof of the last proposition we have seen that operators $f\left(BD\right)$ do not regularize the null space $\kers{p}{BD}$ whenever $f\left(0\right)\neq 0$ and $\kers{p}{BD}\neq \left\{0\right\}$. In the special case of block form operators $\displaystyle BD=\begin{pmatrix} 0 & -\div \\ A\nabla & 0 \end{pmatrix}$ as in \cite[Section 6]{Aus-Sta-1} we have 
\begin{align*}
\kers{p}{BD}=\left\{u=\left(0,g\right)\in L^{p}\left(\mathbb{R}^{n};\mathbb{C}^{m}\oplus[\mathbb{C}^{m}\otimes\mathbb{C}^{n}]\right)\, ;\, 
\div g=0\right\}.
\end{align*}
The interpretation in this special case is that $f\left(BD\right)$ does not regularize the tangential part of functions $\left(0,g\right)$ which satisfy $\div g=0$. In this connection, we note also that the space 
\begin{align*}
\kers{p}{\div}=\left\{g\in L^{p}\left(\mathbb{R}^{n};\mathbb{C}^{m}\otimes\mathbb{C}^{n}\right)\, ;\, 
\div g=0\right\}.
\end{align*}
is invariant by rescaling.
\end{example}

\begin{corollary}[Null space equal zero]\label{corollary3.14}
Let $T\in \left\{DB,BD\right\}$. Denote by $\omega:=\omega_{DB}=\omega_{BD}$ the type of bisectoriality and let $\nu\in \left(\omega,\frac{\pi}{2}\right)$. Further, suppose there exists $r\in \mathcal{I}_{D,B}$ such that $\kers{r}{T}=\left\{0\right\}$.
\begin{enumerate}
\item For all $p,q\in \mathcal{I}_{D,B}$ satisfying $p\leq q$ the semigroup $\left\{e^{-t|T|}\right\}_{t>0}$ is $L^{p}-L^{q}$ bounded.
\item For all $p,q\in \mathcal{I}_{D,B}$ such that $0\leq\frac{n}{p}-\frac{n}{q}<1$ the family $\left\{\left(Id+itT\right)^{-1}\right\}_{t>0}$ satisfies $L^{p}-L^{q}$ off-diagonal estimates of order $K$ for each $K\in \left[0,\infty\right)$.
\end{enumerate}
\end{corollary}
\begin{proof}
The easy details are left to the reader.
\end{proof}

 Corollary \ref{corollary3.14} really shows the link between $L^p-L^q$ estimates for $p<q$ and the triviality of the null space.

\begin{example}\cite[Proposition 3.11]{Aus-Sta-2}
Let $T=BD$. Moreover, suppose $n=1$ and $\hat{D}\left(\xi\right)$ is invertible for all $\xi\neq 0$. Then we have $\mathcal{I}_{D,B}=\left(1,\infty\right)$ and for all $p\in \mathcal{I}_{D,B}$, $\kers{p}{BD}=\kers{p}{D}=\left\{0\right\}$. The reader checks that the proof goes through for $B$ satisfying only coercivity instead of strict accretivity. In fact, the coercivity of $B$ suffices to deduce invertibility of $B\in L^{\infty}\left(\mathbb{R},\bnd\left(\mathbb{C}^{N}\right)\right)$ from the Lebesgue differentiation theorem.
\end{example}

\subsection{Analytic extensions}
Sometimes one is interested in complex times for the results above. So, for appropriate $z\in \mathbb{C}\backslash\left\{0\right\}$ and $f\in \Hinfzer{\nu}$ we define $f_{z}\left(\lambda\right):=f\left(z\lambda\right)$, $\lambda\in \opbisec{\nu}\cup\left\{0\right\}$ and treat this topic in the next remark. 
\begin{remark}[Analytic Extension]
One can extend the results ...
\begin{itemize}
\item ... in Proposition \ref{proposition3.3} to families $\left\{g\left(T\right)\psi_{z}\left(T\right)\right\}_{z\in \opbisec{\beta}}$, $\beta\in \left[0,\frac{\pi}{2}-\omega\right)$, provided there exists $\epsilon\in \left(0,\frac{\pi}{2}-\omega-\beta\right)$ such that $\psi\in\Psidecpar{\sigma}{\tau}{\omega+\epsilon+\beta}$ and $g \in \Hinf{\omega+\epsilon+\beta}$,
\item ... in Example \ref{example3.7} to families  $\left\{\left(izT\right)^{\alpha}\left(Id+izT\right)^{-M}\right\}_{z\in \opbisec{\beta}}$, $\beta\in \left[0,\frac{\pi}{2}-\omega\right)$,
\item ... in Corollary \ref{corollary3.13} to families $\left\{g\left(T\right)f_{z}\left(T\right)\right\}_{z\in \opbisec{\beta}}$, $\beta\in \left[0,\frac{\pi}{2}-\omega\right)$, provided there exists $\epsilon\in \left(0,\frac{\pi}{2}-\omega-\beta\right)$ such that $g,f \in \Hinfzer{\omega+\epsilon+\beta}$ and $f$ satisfies $|f\left(\lambda\right)|\lesssim \inf\left\{1,|\lambda|^{-M}\right\}$ for all $\lambda\in \opbisec{\omega+\epsilon+\beta}$. In particular, the family $\left\{e^{-z|T|}\right\}_{z\in \opsec{+}{\beta}}$, $\beta\in \left[0,\frac{\pi}{2}-\omega\right)$ is $\crans{p}{T}-\crans{q}{T}$ bounded.
\item ... in Corollary \ref{corollary3.14} to the families $\left\{e^{-z|T|}\right\}_{z\in \opsec{+}{\beta}}$ and $\left\{\left(Id+izT\right)^{-1}\right\}_{z\in \opbisec{\beta}}$, $\beta\in \left[0,\frac{\pi}{2}-\omega\right)$.
\end{itemize}
\end{remark}
\begin{proof}
One can adapt the strategies in \cite[Chapter 3.6]{Aus-1}. Details are left to the interested reader.
\end{proof}

\section{An Application}
Here, we will essentially follow \cite[Section 6.2]{Aus-1}\footnote{We mention there are some inaccuracies in this argument that our argument fixes.} to prove $L^{p}-L^{q}$ estimates for the fractional operators $|DB|^{-\alpha}$ and $|BD|^{-\alpha}$ with some simplifications in the final limiting argument. We begin with the definition of $|T|^{-\alpha}$ for $T\in\left\{DB,BD\right\}$ and $\alpha\in \mathbb{C}$, $0<\Re \alpha<\infty$. 
Fix $p,q$ with $\Re \alpha=\frac{n}{p}-\frac{n}{q}$. 
For  $h\in \rans{p}{T}\cap \crans{q}{T}$, define
\begin{align}\label{EQ.4.1}
|T|^{-\alpha}h
:=\frac{1}{\Gamma\left(\alpha\right)}\int_{0}^{\infty}t^{\alpha-1}e^{-t|T|}hdt
=\frac{1}{\Gamma\left(\alpha\right)}\lim_{\left(\epsilon,R\right)\rightarrow \left(0,\infty\right)}\int_{\epsilon}^{R}t^{\alpha-1}e^{-t|T|}hdt,
\end{align}
and observe that the improper Riemann integral converges in the strong sense in $\crans{q}{T}$ with respect to $L^{q}$ topology. Indeed, convergence at 0 follows from $h\in \crans{q}{T}$ and $\Re \alpha>0$, and convergence at $\infty$ follows from  $\Re \alpha=\frac{n}{p}-\frac{n}{q}$ and 
$\|e^{-t|T|}h\|_{L^q}\lesssim t^{\frac{n}{q}-\frac{n}{p}}t^{-1}$ by writing $h=Tf$ with $f\in L^p$ and using $\crans{p}{T}$ to $\crans{q}{T}$ boundedness of the semigroup. 

The result  we want to prove in this section is the following Hardy-Littlewood-Sobolev inequality, which is the analogue to \cite[Section 6.2]{Aus-1}.

\begin{theorem}[Hardy-Littlewood-Sobolev inequality]\label{theorem4.1}
Let $T\in \left\{DB,BD\right\}$. 
Suppose $p_{-}\left(D,B\right)<p<q<p_{+}\left(D,B\right)$. Then $|T|^{-\alpha}$ has a bounded extension from $\crans{p}{T}$ to $\crans{q}{T}$ whenever $\Re \alpha=\frac{n}{p}-\frac{n}{q}$.
\end{theorem}

\begin{proof}
Fix $\Re \alpha:=\frac{n}{p}-\frac{n}{q}$. Set $\mathbb{T}_{\epsilon,R}h:=\frac{1}{\Gamma\left(\alpha\right)}\int_{\epsilon}^{R}t^{\alpha-1}e^{-t|T|}hdt$ 
for $h\in \crans{p}{T}$ and $0<\epsilon<R<\infty$. The first step is to establish  the weak type $p-q$ estimate for $\mathbb{T}_{\epsilon,R}$ from  $\crans{p}{T}$ to $L^{q,\infty}$ uniformly in $\epsilon,R$.  Choose $q_{0},q_{1}$ with $p<q_{0}<q<q_{1}<p_{+}\left(D,B\right)$. Since the semigroup $e^{-t|T|}$ is bounded from $\crans{p}{T}$ to $\crans{q_{0}}{T}$ and to $\crans{q_{1}}{T}$ we get for $h\in \crans{p}{T}$ with $||h||_{L^{p}}=1$, 
\begin{align*}
&||\int_{b}^{R}t^{\alpha-1} e^{-t|T|}h\, dt||_{L^{q_{1}}}\leq \int_{b}^{R}t^{\frac{n}{p}-\frac{n}{q}-1}||e^{-t|T|}h||_{L^{q_{1}}}\,dt\\
&\leq C\int_{b}^{R}t^{\frac{n}{p}-\frac{n}{q}-1}t^{\frac{n}{q_{1}}-\frac{n}{p}}\,dt||h||_{L^{p}}\leq Cb^{\frac{n}{q_{1}}-\frac{n}{q}},
\end{align*}
and similarly
\begin{align*}
&||\int_{\epsilon}^{b}t^{\alpha-1}e^{-t|T|}h \, dt||_{L^{q_{0}}}\leq Cb^{\frac{n}{q_{0}}-\frac{n}{q}},
\end{align*}
uniformly for $\epsilon, b,  R$ such that $0<\epsilon<b<R<\infty$. Hence, 
for all $\lambda>0$ we get from Tchebycheff's inequality 
\begin{align*}
\left|\left\{|\mathbb{T}_{\epsilon,R}h|>\lambda\right\}\right|&\leq \left|\left\{\left|\int_{b}^{R}t^{\alpha-1} e^{-t|T|}h \, dt\right| >\frac{\lambda}{2}\right\}\right|
+\left|\left\{\left|\int_{\epsilon}^{b}t^{\alpha-1}e^{-t|T}h\, dt\right| >\frac{\lambda}{2}\right\}\right|\\
&\leq C\lambda^{-q_{1}}b^{q_{1}\left(\frac{n}{q_{1}}-\frac{n}{q}\right)}
+C\lambda^{-q_{0}}b^{q_{0}\left(\frac{n}{q_{0}}-\frac{n}{q}\right)}.
\end{align*}
Thus, if we choose $b^{-\frac{n}{q}}=\lambda$, we get 
\begin{align*}
&&\left|\left\{|\mathbb{T}_{\epsilon,R}h|>\lambda\right\}\right|\leq C\lambda^{-q}&& \forall \lambda\in \big(R^{-\frac{n}{q}},\epsilon^{-\frac{n}{q}}\big). 
\end{align*}
Similarly, one proves in the case $\lambda\leq R^{-\frac{n}{q}}$
\begin{align*}
\left|\left\{|\mathbb{T}_{\epsilon,R}h|>\lambda\right\}\right|
\leq C\lambda^{-q_{0}}R^{q_{0}\left(\frac{n}{q_{0}}-\frac{n}{q}\right)}
\leq C\lambda^{-q}
\end{align*}
and in the case $\lambda\geq \epsilon^{-\frac{n}{q}}$ 
\begin{align*}
\left|\left\{|\mathbb{T}_{\epsilon,R}h|>\lambda\right\}\right|
\leq C\lambda^{-q_{1}}\epsilon^{q_{1}\left(\frac{n}{q_{1}}-\frac{n}{q}\right)}
\leq C\lambda^{-q}
\end{align*}
to deduce the inequality 
\begin{align*}
&\left|\left\{|\mathbb{T}_{\epsilon,R}h|>\lambda\right\}\right|\leq C\lambda^{-q}, &&\forall \lambda\in \left(0,\infty\right).
\end{align*}

The second step is to proceed by real interpolation. Observe that the spaces  $\crans{p}{D}$ are real interpolation spaces for $1<p<\infty$ (This is shown in \cite{Hyt-McI}). For $p_{-}\left(D,B\right)<p<p_{+}\left(D,B\right)$, we have $\crans{p}{DB}=\crans{p}{D}$ and that $\crans{p}{DB}$ and $\crans{p}{BD}$ are similar spaces under multiplication by $B$ (Remark \ref{rem2.8}). Thus, the real interpolation property holds for  $\crans{p}{T}$ when  $p_{-}\left(D,B\right)<p<p_{+}\left(D,B\right)$ for  $T\in \left\{DB,BD\right\}$. Consider now a pair $(p,q)$ with $\Re \alpha=\frac{n}{p}-\frac{n}{q}$ and $p_{-}\left(D,B\right)<p<q<p_{+}\left(D,B\right)$. It is possible to pick two pairs $(p_{0},q_{0})$ and $(p_{1}, q_{1})$ with the same properties and, in addition, $p_{0}<p<p_{1}$ and $q_{0}<q<q_{1}$. By real interpolation, the weak type $p_{i}-q_{i}$ estimates yield the strong type $p-q$ estimate, in the sense that $ \mathbb{T}_{\epsilon,R}$ maps $\crans{p}{T}$ to $L^q$, uniformly over $0<\epsilon<R<\infty$. 

The last step is a limiting argument. Assume  $h\in \rans{p}{T}\cap \crans{q}{T}$. We know that  $\mathbb{T}_{\epsilon,R}h$ converges strongly to $|T|^{-\alpha}h$ in $\crans{q}{T}$ by construction as $\epsilon \to 0$ and $R\to \infty$. 
As we just showed 
$$\sup_{0<\epsilon<R<\infty}||\mathbb{T}_{\epsilon,R}h||_{L^{q}}\leq C||h||_{L^{p}},$$ 
we deduce 
\begin{align*}
|||T|^{-\alpha}h||_{L^q}
\leq C||h||_{L^{p}}.
\end{align*}
By density of $\rans{p}{T}\cap \crans{q}{T}$ in $\crans{p}{T}$ for the $L^p$ topology (Lemma \ref{lemma2.8}, item (3)), $|T|^{-\alpha}$ has a bounded extension from $\crans{p}{T}$ to $L^q$. To see it maps into  $\crans{q}{T}$, we observe that  for  $h\in \rans{p}{T}\cap \crans{q}{T}$, $|T|^{-\alpha}h$ is by construction the limit in $L^q$ of elements in $\crans{q}{T}$. Thus this remains by density for all $h\in   \crans{p}{T}$. 
\end{proof}

\section{Acknowledgments} This work is part of the forthcoming PhD thesis of the author. The author was partially supported by the ANR project ``Harmonic Analysis at its Boundaries``, ANR-12-BS01-0013-01. He thanks Pascal Auscher for many discussions and suggestions on this article.

\end{document}